\documentclass{amsart}
\usepackage{xcolor}
\usepackage{amssymb,latexsym,amsmath,extarrows}
\usepackage{graphicx,mathrsfs}

\newtheorem{theorem}{Theorem}[section]
\newtheorem{lemma}[theorem]{Lemma}

\newtheorem{conjecture}[theorem]{Conjecture}

\newcommand{\e}{\varepsilon}

\newcommand{\f}{\frac}

\newcommand{\wh}{\widehat}

\newcommand{\Id}{{\bf 1}}

\begin{document}

\title{A remark on the mean value inequalities related to Waring's problem}
\author{Xiaochun Li}

\address{
Xiaochun Li\\
Department of Mathematics\\
University of Illinois at Urbana-Champaign\\
Urbana, IL, 61801, USA}

\email{xcli@math.uiuc.edu}

\date{}

\begin{abstract}
Assuming some pointwise estimates on certain Weyl's sum, 
we prove the sharp estimates of the mean value associated to the following exponential sum 
$$
\sum_{n=1}^N 
e^{2\pi i tn^d +2\pi i xn}\,.
$$ 
\end{abstract}

\maketitle

\section{Introduction}
\setcounter{equation}0

It is conjectured in number theory that for any $\e>0$, there is a constant $C_{d, \e}>0$ obeying 
\begin{equation}\label{eqn1}
\bigg| \sum_{n=1}^N e^{ 2\pi i P(n)} \bigg|\leq C_{d, \e} N^{1+\e} \bigg(\frac{1}{q} + \frac{q}{N^d} \bigg)^{\frac{1}{d}}\,,
\end{equation}
for any $N\in\mathbb N$, if $P(x)=\sum_{j=1}^d \alpha_j x^j$ with $\big| \alpha_d -a/q\big|\leq 1/q^2$ and $(a, q)=1$. 
It turns out that it is very challenging and even a small improvement of existing bounds would be interesting. 
Weyl's differencing method yields a non-trivial estimate, 
\begin{equation}\label{eqn2}
\bigg| \sum_{n=1}^N e^{ 2\pi i P(n)} \bigg|\leq C_{d, \e} N^{1+\e}\bigg( \frac{1}{q}+\frac{1}{N}+\frac{q}{N^d}\bigg)^{\frac{1}{2^{d-1}}}\,,
\end{equation}
if the leading coefficient $\alpha_d$ of the polynomial $P$ satisfies $|\alpha_d-a/q|\leq 1/q^2$ with $(a,q)=1$. 
An exciting breakthrough was made by Vinogradov \cite{Vino}, who noticed that the pointwise problem can be reduced to the following mean value estimations, for $p>d(d+1)$,  
\begin{equation}\label{eqn3}
\int_{\mathbb T^d} \bigg|\sum_{n=1}^N a_n e^{2\pi i(x_1n+\cdots + x_d n^d)}\bigg|^{2p} dx\lesssim 
N^{p-d(d+1)/2} \bigg(\sum_{n=1}^N \big|a_n\big|^2\bigg)^{\f12}\,.
\end{equation}
Vinogradov  proved a result that is slightly weaker than (\ref{eqn3}) but enough to strengthen the 
Weyl estimates (\ref{eqn2}).  
The complete solution to the Vinogradov mean value theorem was finally accomplished   
by Bourgain, Demeter and Guth \cite{BDL},  using the decoupling theory,  and by Wooley \cite{Woo-2} via the efficient congruencing method.  Following from (\ref{eqn3}),  one obtains that 
\begin{equation}\label{eqn4}
\bigg| \sum_{n=1}^N e^{ 2\pi i P(n)} \bigg|\leq C_{d, \e} N^{1+\e}\bigg( \frac{1}{q}+\frac{\log q}{N}+\frac{q\log q}{N^d}\bigg)^{\frac{1}{d(d-1)}}\,,
\end{equation}
where the leading coefficient of $P$ obeys the same condition as in  (\ref{eqn2}). 
(\ref{eqn4}) improves Weyl's estimates  for $d\geq 6$. The Vinogradov method becomes powerless 
when $q$ stays between $N^{d-1}$ and $N^d$.  The true magnitude of the Weyl sum remains mysterious. 
Methods available in the study of the exponential sums are limited. Each of them has its own weakness. 
It is hard to believe that (\ref{eqn1})
could be completely solved by any existing method, such as iterating Process A (the differencing of Weyl and/or van der Corput) and  Process B (the Poisson summation plus the stationary phase), Bombieri-Iwaniec's method 
\cite{B-Iwan}, etc. \\

 We are particularly interested in the following weak version of the conjecture.

\begin{conjecture}\label{con-1}
Let $d\in \mathbb N$, $N\in \mathbb N$ and $q$ be a prime number lying in $[N^{d/2}, N^d]$.  If $\big|t-a/q\big| \leq 1/q^2$ and $1\leq a\leq q-1$, then  
\begin{equation}\label{eqn5}
\bigg| \sum_{n=1}^N e^{ 2\pi i t n^d + 2\pi i x n } \bigg|\leq C_{d, \e} N^\e q^{1/d}\,. 
\end{equation}
\end{conjecture}
 
Because $q$ is much larger than $N$, (\ref{eqn5}) is equivalent to
\begin{equation}\label{eqn5'}
\bigg| \sum_{n=1}^N e^{ 2\pi i \frac{a}{q} n^d + 2\pi i x n } \bigg|\leq C_{d, \e} N^\e q^{1/d}\,, 
\end{equation}
where $q\in [N^{d/2}, N^d]$ and $(a, q)=1$. 
Hence, in Conjecture \ref{con-1}, we encounter an incomplete exponential sum over a finite field. The complete sums over finite fields were solved completely by Deligne \cite{De}.  However, the incomplete 
sum seems to be extremely hard, and it is not clear how to use the algebraic structure of the fields for it. 
We noticed that Conjecture \ref{con-1} leads to a large improvement on Waring's problem by a theorem stated below. \\

Analogous to the Vinogradov mean value, for any $p>0$, we define 
\begin{equation}
S(N, p):=  \int_{\mathbb T^2} \bigg| \sum_{n=1}^{N} e^{2\pi i tn^d +2\pi i xn}\bigg|^{2p} dx dt \,.
\end{equation}

It is well-known that such a mean value is closely related to the Waring problem.  
It is natural to ask if the following estimates are true for it. 

\begin{conjecture}\label{con-2}
Let $p\geq d+1$.  For any $\e>0$,  there is a constant $C_{d, p, \e}>0$ obeying 
 \begin{equation}\label{eqn8}
S(N, p) \leq C_{d, p, \e}
N^{2p-d-1+\e}\,.
\end{equation}
\end{conjecture}

It is not difficult to build up (\ref{eqn8}) for large $p$. For instance, 
in \cite{Hua}, Hua gave an arithmetic argument to provide an affirmative answer for large $p\sim 2^d$.
Hua's lemma can be proved in an alternative way by Hu and the author in \cite{H-L2}. In \cite{D-Lai}, Ding and Lai have shown that the Vinogradov mean value theorem implies Conjecture {\ref{con-2}} for $p\geq d(d+1)$.  It will be very interesting if (\ref{eqn8}) can be established for $p\sim d$, because of the direct application in Waring's problem (see \cite{Vino}). \\

Our main theorem in this paper is

\begin{theorem}\label{thm1}
Conjecture \ref{con-1} implies Conjecture \ref{con-2} for all $p\geq d+1$. 
\end{theorem}

From this theorem, we see that the incomplete sum estimates (\ref{eqn5'}) over finite fields yield a nice improvement in Waring's problem.

\section{Technical lemmas associated to the exponential sum}

Let $K_N $ be a kernel defined by
$$
 K_N(x, t) = \sum_{n=1}^N e^{2\pi i t n^d +2\pi i xn}\,. 
$$

Let $Q$ be a positive number and $\phi$ be a standard bump function supported on $[1/200, 1/100]$. 
We define 
\begin{equation}
\Phi(t) = \sum_{\substack{Q\leq q \leq 3Q\\ q\, {\rm prime}} } \sum_{a=1}^{q-1} \phi\bigg( \frac{t-a/q}{1/q^2}\bigg)\,.
\end{equation}
$\Phi$ is supported on $[0,1]$ and can be extended to other intervals periodically to obtain a periodic function on $\mathbb T$. 
We still use $\Phi$ to denote the periodic extension of $\Phi$.   Then it is easy to see that
\begin{equation}\label{size-Phi0}
\wh\Phi(0) = \sum_{\substack{Q\leq q \leq 10Q\\ q\, {\rm prime}} }\sum_{a=1}^{q-1} \frac{\mathcal F_{\mathbb R}{\phi}(0)}{q^2}
\sim   \frac{1}{\log Q} \mathcal F_{\mathbb R}{\phi}(0)
\end{equation}
if $Q$ is sufficiently large, by the prime number theorem. Here $\mathcal F_{\mathbb R} f$ denotes Fourier transform of a 
function $f$. 

\begin{lemma}\label{lem-P-k}
Let $N^{d/2}\leq Q\leq N^d$. 
For the periodic extension $\Phi$, we have, for $k\neq 0$, 
\begin{equation}
\big| \wh \Phi(k) \big| \lesssim Q^{-1}\,. 
\end{equation}
\end{lemma}

\begin{proof}
Notice that 
\begin{equation}
\wh \Phi(k) = \sum_{ \substack{Q\leq q \leq 3Q\\ q\, {\rm prime}}}\sum_{a=1}^{q-1} \frac{1}{q^2} e^{-2\pi i \frac{a}{q} k} 
\mathcal F_{\mathbb R}{\phi}(k/q^2)\,,
\end{equation}
which equals to 
$$
\sum_{ \substack{Q\leq q \leq 3Q\\ q\, {\rm prime}\\
   q|k}}\sum_{a=1}^{q} \frac{1}{q^2} e^{-2\pi i \frac{a}{q} k} 
\mathcal F_{\mathbb R}{\phi}(k/q^2)
  + \sum_{ \substack{Q\leq q \leq 3Q\\ q\, {\rm prime}} }\frac{1}{q^2} \mathcal F_{\mathbb R}{\phi}(k/q^2)\\
\lesssim  Q^{-1}\,,
$$
as desired. 
\end{proof}

We now define 
\begin{equation}
 K_{1, Q}(x, t) = \frac{1}{\wh\Phi(0)} K_N(x,t) \Phi(t)
\end{equation}
and 
\begin{equation}
 K_{2, Q}= K_N - K_{1, Q}\,. 
\end{equation}

\begin{lemma}\label{lemK-1}
Under the hypothesis of Conjecture \ref{con-1},  we have 
\begin{equation}
\big\| K_{1, Q}\|_\infty\lesssim N^\e Q^{1/d}\,. 
\end{equation}
\end{lemma}

\begin{proof}
It follows immediately from (\ref{size-Phi0}) and the pointwise estimates in Conjecture \ref{con-1}. 
\end{proof}

\begin{lemma}\label{lemK-2}
\begin{equation}
\big\|\wh {K_{2, Q}}\big\|_\infty \lesssim \frac{\log Q}{Q}\,. 
\end{equation}
\end{lemma}

\begin{proof}
Represent the periodic function $\Phi$ as its Fourier series and then we end up with 
$$
 K_{2, Q}(x, t) = -\frac{1}{\wh\Phi(0)} \sum_{k\neq 0} \wh \Phi(k) e^{2\pi i kt} K_N(x,t)\,. 
$$
Henceforth its Fourier coefficient is 
$$
\wh{K_{2, Q}}(n_1, n_2) =-\frac{1}{\wh\Phi(0)}\sum_{k\neq 0}\wh\Phi(k) \Id_{\{k: k=n_2-n_1^d\}} (k)\,.
$$
Here $(n_1, n_2)\in\mathbb Z^2$ and recall that $\Id_A$ is the indicator function of a measurable set $A$. 
We then see that 
\begin{equation}
\wh{K_{2, Q}}(n_1, n_2)=  \left\{ \begin{array}{rcl}
      0    & {\rm if} & n_2=n_1^d\, ;  \\
       -\frac{\wh\Phi(n_2-n_1^d)}{\wh\Phi(0)}   & {\rm if} &  n_2\neq n_1^d\,.
\end{array}\right.
\end{equation}
The desired estimate follows from Lemma \ref{lem-P-k} and (\ref{size-Phi0}). 
\end{proof}

\section{Proof of Theorem \ref{thm1}}

We define a level set associated to the kernel $K_N$ by 
\begin{equation}
 G_\lambda = \big\{ (x, t)\in \mathbb T^2: \big|K_N(x, t)\big| >\lambda\big\}\,,
\end{equation}
for any $\lambda>0$.  Let 
\begin{equation}
 f=\Id_{G_\lambda} K_N/ \big|K_N|\,. 
\end{equation}
Then from the definition $G_\lambda$, we get 
\begin{equation}
 \lambda \big| G_\lambda\big| \leq \sum_{n=1}^N \wh f(n, n^d) = \langle f_N, K_N\rangle\,,  
\end{equation}
where $f_N$ is the partial Fourier series given by 
$$
 f_N(x,t)= \sum_{\substack{n_1: |n_1|\leq N\\ n_2: |n_2|\leq N^d}} \wh f(n_1, n_2) e^{2\pi i x n_1}e^{2\pi i t n_2}\,. 
$$
For any $Q$ between $N^{d/2}$ and $N^d$,  we decompose the kernel $K_N$ into $K_{1, Q}$ and $K_{2, Q}$. 
Then we obtain 
\begin{eqnarray*}
 \lambda \big| G_\lambda\big|  &\leq &\big| \langle f_N, K_{1, Q}\rangle \big| +\big| \langle f_N, K_{2, Q}\rangle \big|\\
 &\leq &  \|K_{1, Q}\|_\infty \|f_N\|_{L^1(\mathbb T^2)} + 
 \sum_{\substack{n_1: |n_1|\leq N\\ n_2: |n_2|\leq N^d}} \big|\wh{K_{2, Q}}(n_1, n_2)\big| \big|\wh f(n_1, n_2)\big|\\
 & \leq &    \log N\|K_{1, Q}\|_\infty \|f\|_{L^1(\mathbb T^2)}+ 
   \|\wh {K_{2, Q}}\|_\infty N^{\frac{d+1}{2}} \|f\|_{L^2(\mathbb T^2)} \,.
 \end{eqnarray*}
 Here in the last step we used Cauchy-Schwarz inequality and the fact that $L^1$-norm of the Dirichlet kernel is comparable to $\log N$. 
Using Lemma \ref{lemK-1} and Lemma \ref{lemK-2},  we get the following  estimation for $G_\lambda$, 
\begin{equation}\label{lev}
\lambda \big| G_\lambda\big| \lesssim N^{\e^2} Q^{1/d}\big|G_\lambda\big|+   \frac{N^{\frac{d+1}{2}}\log N}{Q}\big|G_\lambda\big|^{\f12}\,, \forall Q \in [N^{d/2}, N^d]\,. 
\end{equation}
We just need to focus the case when $\lambda \leq N$, since $G_\lambda=\emptyset $ if $\lambda>N$.  
For $N\geq \lambda \geq 100N^{\f12+\e^2}$, we choose $Q\in [N^{d/2}, N^d]$ obeying $ Q^{1/d}=N^{-\e^2}\lambda/100$. 
Then (\ref{lev}) becomes 
\begin{equation}
 \lambda \big| G_\lambda\big| \lesssim  \frac{\lambda}{100}\big|G_\lambda\big|+  
  \frac{N^{\frac{d+1}{2}+\frac{\e}{10}}}{\lambda^d}\big|G_\lambda\big|^{\f12}\,. 
 \end{equation}
 Henceforth, we end up with 
 \begin{equation}\label{l-l}
   \big| G_\lambda\big| \lesssim  \frac{N^{d+1+\e/5}}{\lambda^{2d+2}}\,. 
 \end{equation}
We now turn to the estimate of $S(N, d+1)$.  It is easy to see that 
$$
S(N, d+1)\leq \int_{100 N^{\f12+\e^2}} \lambda^{2d+1}\big|G_\lambda\big| d\lambda 
+ \int_0^{100 N^{\f12+\e^2}} \lambda^{2d+1} \big|G_\lambda\big| d\lambda\,. 
$$
Using (\ref{l-l}), the first term above is bounded by $N^{d+1+\e}$.  Notice that 
$|G_\lambda|\lesssim N/\lambda^2$, which is a consequence of $\|K_N\|_{L^2(\mathbb T^2)}=\sqrt{N}$. 
Then we can estimate the second term by
\begin{equation}
\int_0^{100 N^{\f12+\e^2}} \lambda^{2d+1} \big|G_\lambda\big| d\lambda\lesssim  N\int_0^{100 N^{\f12+\e^2}}
\lambda^{2d-1} d\lambda\lesssim N^{d+1+\e}\,.
\end{equation}
Combining these two estimations, we end up with
\begin{equation}
S(N, d+1)\lesssim N^{d+1+\e}\,,
\end{equation}
as desired.  Therefore, we complete the proof of Theorem \ref{thm1}. 

\vspace{0.5cm}

\noindent
{\bf Acknowledgement} The author is partially supported by Simons foundation.

\end{document}